\documentclass[11pt]{amsart}

\usepackage{amsmath,amssymb,amsthm}
\usepackage{graphicx,color}
\usepackage{kbordermatrix}

\newtheorem{theorem}{Theorem}[section]

\newtheorem{lemma}[theorem]{Lemma}
\newtheorem{corollary}[theorem]{Corollary}

\newtheorem*{theorem*}{Theorem}

\theoremstyle{definition}

\newtheorem*{example}{Example}
\newtheorem*{examplecont}{Example (continued)}
\newtheorem*{remark}{Remark}
\newtheorem*{question}{Question}

\newcommand{\E}{\mathbf{E}}

\newcommand{\glm}{\mathcal{L}}
\newcommand{\degm}{\mathcal{D}}
\newcommand{\adjm}{\mathcal{A}}

\newcommand{\cof}{\operatorname{cof}}
\newcommand{\D}{\Delta}

\newcommand{\1}{\mathbf{1}}

\begin{document}

\title[Determinants, Markov processes, and Kirchhoff's matrix tree
  theorem]{Determinants, their applications to Markov processes, and a
  random walk proof of Kirchhoff's matrix tree theorem}  

\date{\today}

\author{Michael J.~Kozdron}
\address{University of Regina\\
Department of Mathematics and Statistics\\
Regina, SK S4S 0A2 Canada}
\email{kozdron@stat.math.uregina.ca}

\author{Larissa M.~Richards}
\thanks{Research supported in part by the Natural Sciences and Engineering Research Council (NSERC) of Canada.}
\address{University of Regina\\
Department of Mathematics and Statistics\\
Regina, SK S4S 0A2 Canada}
\email{richalar@uregina.ca}

\author{Daniel W.~Stroock}
\address{Massachusetts Institute of Technology\\
Department of Mathematics\\
Cambridge, MA 02139 USA}
\email{dws@math.mit.edu}

\thanks{The research of the first two authors is supported in part by the
  Natural Sciences and Engineering Research Council (NSERC) of Canada: the
  first through a discovery grant and the second through an undergraduate
  student research award.}

\begin{abstract} Kirchhoff's matrix tree theorem is a well-known result
  that gives a formula for the number of spanning trees in a finite, connected graph
  in terms of the graph Laplacian matrix. 
  A closely related result is Wilson's algorithm for putting the uniform distribution on the set of spanning trees. We will show that when one follows Greg Lawler's strategy for proving Wilson's algorithm, Kirchhoff's theorem follows almost immediately after one applies some elementary linear algebra. We also show that the same ideas can be applied to other computations related to general Markov chains and processes on a  finite  state space.\end{abstract}

\maketitle

\section{Introduction}\label{SectIntro}

Markov chains and Markov processes on a finite state space are completely
determined by a matrix.  In the case of chains, it is a transition matrix $P$ whose $(i,j)$ entry specifies the probability
that the chain will go in one step from state $i$ to state $j$.  In the
case of Markov processes, it  is a $Q$-matrix of the form $R(P-I)$, where $R$ is
a diagonal matrix with non-negative entries and $P$ is a transition
 matrix.  The $i$th diagonal entry of $R$ gives the rate
at the process leaves the state $i$ and $P$ gives the distribution of where
it will go when it leaves.

At least in theory, every question that one can ask about the chain or
process can be answered in terms of $P$ or $Q$.  However, in practice, it
is often difficult to write down a transparent expression that gives the
answer.  For example, a quantity of probabilistic interest is the
stationary distribution, that is, a distribution that is left invariant by
the chain or process.
Every Markov chain or process on a finite state space admits at least one
stationary distribution, and sometimes it has many.  A necessary and
sufficient condition for it to have precisely one is that there exist a
point that is accessible from every point.  When one has such a situation,
one would like to have a simple expression for this unique distribution
in terms of $P$ or $Q$, and, as an application of the ideas here, we will
give one.

\begin{remark}  The usual procedure for finding stationary distributions is
  to look for solutions to $\pi P\equiv P^*\pi =\pi $ in the case of Markov
  chains and $\pi Q=Q^*\pi =0$ in the case of Markov processes.  Among other
  places, this procedure is discussed in the books~\cite{KS} and~\cite{Norris}.
\end{remark}

A quite different application of our considerations is to Wilson's
algorithm and Kirchhoff's matrix tree theorem.  Given a finite, connected
graph, there are lots of spanning trees (i.e., subgraphs that contain no
cyles and include all vertices).  In  1847, Gustav Kirchhoff~\cite{Kirchhoff}
gave a formula for the number of spanning trees, and in 1996 David Wilson~\cite{Wilson}
gave an algorithm for generating a spanning tree uniformly at random (without knowing the actual number of spanning trees).
  Kirchhoff's formula expresses the number in terms of the graph Laplacian matrix $\mathcal
L\equiv\mathcal D-\mathcal A$, where $\mathcal D$ is the diagonal matrix whose $i$th diagonal entry
is the degree (i.e., the number of vertices to which it is connected by an
edge) of the vertex $i$ and $\mathcal A$ is the adjacency matrix, the matrix whose
$(i,j)$ entry is $1$ if there is an edge between $i$ and $j$ and is $0$
otherwise.  Obviously, $-\mathcal L=\mathcal D(P-I)$ where $P=\mathcal D^{-1}\mathcal A$ is
a transition  matrix, and so $-\mathcal L$ is a $Q$-matrix.
Wilson's algorithm uses the chain determined by the $P$ to explore the
graph.  Since, with probability $1$, his algorithm produces a spanning tree,
and since the probability of its producing any particular one is the same
for all spanning trees, the number of spanning trees must be equal to the
reciprocal of the probability that Wilson's algorithm produces a
particular one.  Thus, one can recover Kirchhoff's result if one can show
that the probability that Wilson's algorithm produces a particular spanning tree is
the reciprocal of Kirchhoff's expression in terms of $\mathcal L$.  Following
a strategy developed by Greg Lawler, we will show how this can be done.

\begin{remark} 
Even today, Kirchhoff's result has to be considered a sophisticated application of matrix algebra, but in 1847, when matrix algebra was in its infancy, it was a remarkable achievement.  Most modern proofs are based on the Cauchy-Binet formula
and have no obvious connection to probability theory; for an easily accessible
account, see~\cite{JvdB}. Lawler's proof~\cite{LawLimic} that Wilson's algorithm works is very different from
Wilson's own proof~\cite{Wilson}; see also~\cite{Grimmett} for a detailed exposition of Wilson's technique.  Our proof is based on the same idea as Lawler's.
\end{remark}

\begin{example} In order to illustrate the concepts, notation, and proof of the matrix tree theorem via Wilson's algorithm, we will work with the following example. Consider the graph $\Gamma$ shown below having vertex set $\{x_1,x_2,x_3,x_4,x_5,x_6\}$  and graph Laplacian matrix $\glm = \degm-\adjm$ as given.
\begin{center}
\begin{minipage}{\textwidth}
  \centering
  \raisebox{-0.5\height}{\includegraphics[height=1.35in]{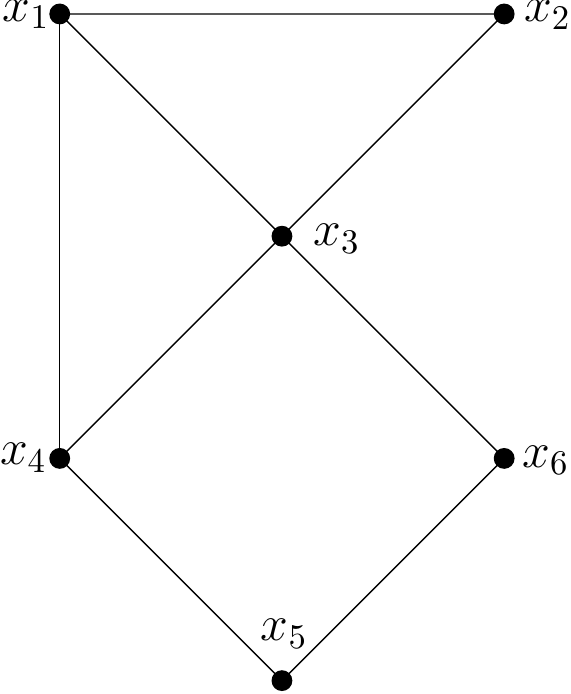}}
  \hspace*{.5in}
$\glm =
\kbordermatrix{
&x_1 &x_2 &x_3  &x_4 &x_5 &x_6 \\
  x_1 &3  &-1 &-1 &-1  &0 &0\\
  x_2 &-1  &2 &-1 &0  &0 &0\\
  x_3 &-1   &-1 &4 &-1  &0 &-1\\
  x_4 &-1 &0 &-1 &3  &-1 &0\\
  x_5 &0   &0 &0 &-1  &2 &-1\\
  x_6 &0 &0 &-1 &0  &-1 &2
}
$
\end{minipage}
\end{center}
It can be easily determined that this graph has 29 spanning trees if one considers the degree of $x_3$ in the spanning tree: $\deg(x_3)=4$ for 2 spanning trees, $\deg(x_3)=3$ for 10 spanning trees, 13 spanning trees have $\deg(x_3)=2$, and 4 spanning trees have $\deg(x_3)=1$.
\end{example}

\section{Two theorems about determinants}\label{SectLinAlg}

In this section, we will be dealing with the vector space ${\mathbb C}^N$ and will be
using the Hermitian inner product $\langle x,y\rangle=x_1\overline{y_1} + \cdots + x_N\overline{y_N}$ and
norm $|x|^2=\langle x,x\rangle$.   Given an $N \times N$ matrix $M\in{\mathbb C}^N\otimes {\mathbb C}^N$
and a subset $\D \subsetneq \{1, \ldots,
N\}$ with $n$ elements, denote by $M^\Delta $ the $(N-n)\times (N-n)$
matrix obtained from $M$ by removing the rows and columns corresponding to the
indices in $\Delta $.  Assuming that $\det[M]\neq0$, Cramer's rule states that
the $(i,i)$ entry of $M^{-1}$ is
\begin{equation}\label{eq1} (M^{-1})_{ii}
  = \frac{\det[M^{\{i\}}]}{\det[M]} \end{equation}
for all $i=1,\ldots,N$.
The following theorem now follows immediately from~\eqref{eq1} by
induction.
 
 \begin{theorem}\label{detthm1} Let $M$ be a non-degenerate $N \times N$
   matrix. Suppose that $(\sigma(1), \ldots, \sigma(N))$ is a permutation
   of $(1,\ldots, N)$. Set $\D_1 = \emptyset$ and, for $j = 2,\ldots, N$,
   let $\D_{j} = \D_{j-1} \cup \{\sigma(j-1)\} = \{\sigma(1), \ldots,
   \sigma(j-1)\}$. If $M^{\D_j}$ is non-degenerate for all $j=1,\ldots, N$,
   then
\begin{equation}\label{eq2}  \prod_{j=1}^n
     (M^{\D_j})^{-1}_{\sigma(j)\sigma(j)} =\det[M]^{-1} .\end{equation} \end{theorem}

Obviously, the interest in~\eqref{eq2} is not that it gives an efficient way
of computing $\det[M]^{-1}$, but that it shows that the product on the left
is independent of the permutation $\sigma$. 

\begin{example}
We illustrate how to use the notation of Theorem~\ref{detthm1} to do a
computation. Suppose that $M$ is the non-degenerate $3\times 3$ matrix  
$$M=\kbordermatrix{
      &1  &2 &3\\
1  &9/10  &2/5  &-1/10 \\
2 &1/10 &-2/5  &-9/10 \\
3  &-1/5  &4/5  &-1/5}
\;\;\; \text{so that} \;\;\;
M^{-1} 
=\kbordermatrix{
      &1   &2 &3 \\
1 &1&0 &-1/2  \\
2 &1/4 &-1/4  &1 \\
3  &0 &-1  &-1/2 }.
$$
Observe that $\det[M]^{-1}= (4/5)^{-1}=5/4$. We will now  calculate this
determinant using~\eqref{eq2}. 
Let $\sigma$ be any permutation of $\{1,2,3\}$, say $\{2,3,1\}$, so that
$\D_1=\emptyset$, $\D_2= \{2\}$, $\D_3=\{2,3\}$. We now find $(M^{\D_1})^{-1} = M^{-1}$,
$$
(M^{\D_2})^{-1}
=\left(\kbordermatrix{
     &1  &3\\
1  &9/10   &-1/10 \\
3  &-1/5   &-1/5} \right)^{-1}=
\kbordermatrix{
      &1   &3\\
1  &1  &-1/2  \\
3  &-1  &-9/2 }\;\;\;
\text{and} \;\;\; (M^{\D_3})^{-1}
=\kbordermatrix{
       &1 \\
1  &10/9 } 
$$
and so
$$\prod_{j=1}^3 (M^{\D_j})^{-1}_{\sigma(j)\sigma(j)} = (M^{\D_1})^{-1}_{22} (M^{\D_2})^{-1}_{33}
(M^{\D_3})^{-1}_{11} =-\frac{1}{4}\cdot-\frac{9}{2}\cdot\frac{10}{9} = \frac{5}{4}.$$
\end{example}

The next theorem relies on the Jordan-Chevalley decomposition~\cite{FIS} of a matrix
into its semi-simple and nilpotent parts.  Namely, what we need to know is
that if $\lambda _1,\ldots,\lambda _N$ are the eigenvalues (equivalently,
the roots, counting multiplicity, of the characteristic polynomial) of $M$, then
there exist matrices $B$ and $C$ and a basis $(v_1,\ldots,v_N)$ for ${\mathbb
  C}^N$ such that $M=B+C$ where $C$ is nilpotent (i.e., $C^N=0$), $Bv_i=\lambda _iv_i$
for $1\le i\le N$, and $B$ commutes with $C$. 

\begin{theorem}\label{detthm2}
Suppose that $M$ is an $N \times N$  matrix with eigenvalues $\lambda_1,
\ldots, \lambda_N$, assume that $\lambda _j=0$ if and only if $j=1$, and set
$\Pi_M =\lambda_2\cdots \lambda_N$, the product of the non-zero eigenvalues
of $M$.  Then there is a unique $v_1^*$ such that $M^*v_1^*=0$ and
$\langle v_1,v_1^*\rangle=1$.  Moreover
\begin{equation}\label{eq3}
\frac{\det[M^{\{i\}}]}{\Pi_M} = (v_1)_i  (\overline{v_{1}^*})_i.
\end{equation}\end{theorem}

\begin{proof}
Let $B$, $C$, and $(v_1, \ldots, v_N)$ be the quanties described above.
Because the kernel of $M$ has the same dimension as that of $M^*$, we know
that $v_1^*$ is uniquely determined by the conditions $M^*v_1^*=0$ and
$\langle v_1,v_1^*\rangle=1$.  Now choose $(v_2,\ldots,v_N)$ so that
$(v_1^*,\dots,v_N^*)$ is the dual basis for $(v_1,\dots,v_N)$ (i.e.,
$(v_i,v_j^*)=\delta _{i,j}$). For $\alpha \notin \{-\lambda_1=0, -\lambda_2, \ldots,
-\lambda_N\}$, 
\begin{equation}\label{nilpinveq1}
(\alpha I + M)^{-1} = (\alpha I +B)^{-1} + \sum_{n=1}^{N-1} (-1)^n(\alpha I +B)^{-(n+1)}C^n.
\end{equation}
To verify this, multiply the right side of~\eqref{nilpinveq1} by
$(\alpha I +B)+C$, expand, and note that necessarily $C^{N}=0$ since $C$
is nilpotent.  Hence, because $Cv_1=Mv_1-Bv_1=0$,
\begin{align*}
(\alpha I + M)^{-1}_{ii} &= \sum_{j'=1}^N  \sum_{j=1}^N   \langle (\alpha
  I + M)^{-1}v_{j'}, v_j^* \rangle (v_j)_i (\overline{v_{j'}^*})_i\\ 
&=\sum_{j=1}^N \frac{ (v_j)_i (\overline{v_{j}^*})_i}{\alpha+\lambda_j} +
  \sum_{n=1}^{N-1} (-1)^n \sum_{j'=2}^N \frac{ (\overline{v_{j'}^*})_i
  }{(\alpha+\lambda_{j'})^{n+1}} \sum_{j=1}^N \langle C^n v_{j'}, v_j^*
  \rangle (v_j)_i.
\end{align*}
If we multiply both sides of the preceding equality by $\alpha$ and let
$\alpha \searrow 0$, then, since  $\lambda_j=0$ iff $j=1$, we see that only
the first term on the right hand side survives; that is, 
$$\lim_{\alpha \downarrow 0} \alpha (\alpha I + M)^{-1}_{ii} = (v_1)_i  (\overline{v_{1}^*})_i.$$
By combining this with~\eqref{eq1}, and observing that
$$\det[\alpha I +M] = \alpha \prod_{j=2}^N (\alpha+ \lambda_j)$$
we get~\eqref{eq3} as required.
\end{proof}

By summing over $i$ in~\eqref{eq3}, we see that
\begin{equation}\label{eq4}
\Pi_M = \sum_{i=1}^N \det[M^{\{i\}}].
\end{equation}
However, this an inefficient way to prove~\eqref{eq4}, which holds without
any assumptions on $M$ other than that $\det[M]=0$.  Indeed, when
$\det[M]=0$, the right side of~\eqref{eq4} is then the constant term in the
polynomial $-\lambda^{-1}\det[M-\lambda I]$.

\begin{remark}
Let $\cof[M]$ be the cofactor matrix for $M$, namely the matrix with
$(i,j)$ entry equal to $(-1)^{i+j}$ times the determinant of the matrix
obtained from $M$ by removing the $i$th column and $j$th row. The same
argument that led to~\eqref{eq3} shows that if $\lambda$ is a simple (i.e., it is different from all the others)
eigenvalue of $M$ with eigenvector $v$, then 
$$\overline{v^*} = \frac{\cof[M-\lambda I] \overline{v}}{|v|^2\Pi_{M-\lambda I}}.$$
In particular, when $M$ is normal and therefore $v^* = |v|^{-2}v$, then $v$ is
an eigenvector for $\overline{M-\lambda I}$ with eigenvalue
$\Pi_{\overline{M-\lambda I}}$. 
\end{remark}

\section{A brief account of Markov chains and processes}\label{SectMark}

As we said, a transition matrix $P$ on a finite state space
$V$ is the key ingredient in the construction of Markov chains.  Namely, a
Markov chain with transition matrix  $P$ is a sequence, best thought of
as a path
$\{X_n:\,n\ge0\}$, of $V$-valued random variables with the property that,
for any $n\ge1$ and $x_0,\ldots,x_n\in V$,
\begin{equation}\label{MP}
\mathbb P\bigl(X_m=x_m\text{ for } 0\le m\le n\bigr)=\mathbb
P\bigl(X_m=x_m\text{ for } 0\le m<n\bigr)P_{x_{n-1}x_n}.\end{equation}
Equivalently, in probabilistic terminology, the conditional probability given $X_0,\ldots,X_{n-1}$ that
$X_n=x_n$  is $P_{X_{n-1},x_n}$.  It is easy to
verify that if one knows the distribution of $X_0$, then the distribution of
$\{X_n:\,n\ge0\}$ is uniquely determined by~\eqref{MP}. We will use the
notation $\mathbb P_x$ to denote the distribution of the chain when $X_0\equiv
x$ and $\mathbb E_x$ to denote expectations with respect to $\mathbb P_x$.
Starting from~\eqref{MP}, it is not hard to show that, for any $n_1$ and $n_2$,
\begin{equation}\label{MP2}\begin{aligned}
\mathbb P&\bigl(X_m=x_m\text{ for } 0\le m\le n_1+n_2\bigr)\\
&\qquad=\mathbb P\bigl(X_m=x_m\text{ for } 0\le m\le n_1\bigr)\mathbb P_{x_{n_1}}
\bigl(X_{n_1+m}=x_m\text{ for } 1\le m\le n_2\bigr)
\end{aligned}\end{equation}

Given a $\Delta \subsetneq V$, let $\xi ^\Delta \equiv\inf\{n\ge0:\,X_n\in\Delta
\}$ be the first time that the chain visits $\Delta $.  Thus $\xi ^\Delta
=\infty $
for paths that never visit $\Delta $.  Say that $\Delta $ is
accessible from $x\notin\Delta $ if there exist an $n\ge1$ and  points $x_0,\ldots,x_n\in
V$ such that $x_0=x$, $x_n\in \Delta $, and $P_{x_{m-1},x_m}>0$ for $1\le
m\le n$. 

\begin{lemma}\label{timelem}  If $\Delta $ is accessible from each
$x\in V$, then $\max_{x\in V}\mathbb E_x[\zeta ^\Delta ]<\infty $.\end{lemma}

\begin{proof}  By the accessibility assumption, we know that, for each
  $x\in V$, there is an $n$ and an $\theta \in(0,1)$ such that $\mathbb P_x(\xi
  ^\Delta >n)\le \theta $, and because $V$ is finite, we can choose one $n$
  and $\theta $ that works simultaneously for all $x\in V$.  Hence, by~\eqref{MP2},
$$\mathbb P_x\bigl(\xi^\Delta >(k+1)n\bigr)=\sum_{y\in V}\mathbb
  P_x\bigl(\xi ^\Delta >kn\;\&\;X_{kn}=y\bigr)\mathbb P_y\bigl(\xi ^\Delta >n\bigr)
\le \theta \; \mathbb P_x\bigl(\xi ^\Delta >kn\bigr).$$
By induction, this means that $\mathbb P_x(\xi ^\Delta >kn)\le\theta ^k$, and
so the asserted result follows.\end{proof}

Continuing under the conditions in Lemma~\ref{timelem}, our next goal is to show
that $(I-P)^\Delta $ is invertible  and that 
\begin{equation}\label{chaintime}\begin{gathered}
\E_x\left[\sum_{n=0}^{\xi ^\Delta }\1_{\{y\}}(X_n)\right]
=\bigl((I-P)^\Delta \bigr)^{-1}_{xy}\;\text{ for }\; x,y\in V\setminus \Delta .\end{gathered}\end{equation}
Perhaps the most elementary way to check this is to first observe that
$$\E_x\left[\sum_{n=0}^{\xi ^\Delta }\1_{\{y\}}(X_n)\right]=\sum_{n=0}^\infty
\mathbb P_x\bigl(X_n=y\;\&\;\xi ^\Delta >n\bigr).$$
Second, note that, because $y\notin\Delta $, we have $\mathbb
P_x\bigl(X_n=y\;\&\;\xi ^\Delta >n\bigr)= 
\mathbb P_x\bigl(X_{n\wedge \xi ^\Delta }=y\bigr)$.  Finally, $\{X_{n\wedge
  \xi^\Delta }:\,n\ge 0\}$ is the chain starting at $x$ determined that
  the transition matrix $P'$ whose $(x',y')$ entry equals $P_{x'y'}$ if
  $x'\in V\setminus D$ and equals $\delta _{x',y'}$ if $x'\in \Delta $.
  Hence, by~\eqref{MP}, $\mathbb P_x\bigl(X_{n\wedge \xi ^\Delta
  }=y\bigr)=(P')^n_{x,y}$, and it is an easy matter to see that
  $(P')^n_{xy}=(P^\Delta )^n_{xy}$.  Thus, we now know that
$$\E_x\left[\sum_{n=0}^{\xi ^\Delta
    }\1_{\{y\}}(X_n)\right]=\sum_{n=0}^\infty (P^\Delta )^n_{xy}.$$
In particular, since the left hand side is dominated by $\E_x[\xi ^\Delta
]$, the series on the right converges.  Finally, knowing that this series
converges, an elementary argument shows that 
$$(I-P^\Delta )\sum_{n=0}^\infty (P^\Delta )^n
=I $$
and therefore that $I+ P^\Delta +(P^\Delta )^2+ \cdots $
is the inverse of $(I-P)^\Delta$.

A closely related consideration is the following.  Given $x\in V$, define
$\{\tau _x^{(m)}:\,m\ge0\}$ inductively so that $\tau _x^{(0)}=0$ and $\tau
^{(m)}_x=\inf\{n>\tau _x^{(m-1)}:\,X_n=x\}$.  If $x\in V\setminus D$, then
it is clear that $\tau _x^{(m)}<\xi ^\Delta $ if and only if
$$\sum_{n=0}^{\xi ^\Delta }\1_{\{x\}}(X_n)>m.$$ 
 At the same time, using~\eqref{MP2}, one can check that $\mathbb P_x(\tau^{(m)_x}<\xi ^\Delta )=
\mathbb P_x(\tau_x^{(m-1)}<\xi ^\Delta )\mathbb P_x(\tau_x^{(1)}<\xi
^\Delta )$, and so $\mathbb P_x(\tau_x^{(m)}<\xi ^\Delta )=\mathbb P_x(\tau
^{(1)}_x<\xi ^\Delta )^m$.  Combining these, we arrive at
\begin{equation}\label{returneq}
\bigl((I-P)^\Delta \bigr)^{-1}_{xx}=\E_x\left[\sum_{n=0}^{\xi ^\Delta    }\1_{\{y\}}(X_n)\right]=\frac1{1-r_\Delta (x)}
\;\;\; \text{where} \;\;\; r_\Delta (x)\equiv\mathbb P_x(\tau ^{(1)}_x<\xi ^\Delta ).
\end{equation}
Analogous results hold for the Markov processes associated with a
$Q$-matrix.  Indeed, if $Q=R(P-I)$ and $\{X_n:\,n\ge0\}$ is a Markov chain
with transition matrix $P$ and initial distribution $\mu $, then a 
Markov process $\{X(t):\,t\ge0\}$ determined by $Q$ with initial
distribution is obtained from $\{X_n:\,n\ge0\}$ by randomizing the time that
it stays at a point.  More precisely, one can always choose $R$ and $P$ so
that $P_{xx}=0$ for all $x\in V$.  With this choice, the unparameterized
paths of $\{X(t):\,t\ge0\}$ are exactly the same as those of
$\{X_n:\,n\ge0\}$.  The difference is that, instead of remaining in a state
for time $1$, if $X(s)=x$, then it stays at $x$ for an exponential holding
time with rate constant $R_x$.  That is, the probability has not left $x$
before some time $t>s$ is $e^{-(t-s)R_x}$.  On the other hand, when it
leaves a point, the distribution of where it goes is exactly the same as
that for the chain.

Fortunately, we do not need to know much about these processes.
Based on the preceding description, one can show that the
probability that the process started from $x$ will be at $y$ at time $t$ is
$(e^{tQ})_{xy}$, where
$$e^{tQ}\equiv\sum_{k=0}^\infty \frac{t^nQ^n}{n!}.$$
Starting from this, it is easy to see that  $\pi $ is a stationary
distribution for the process if and only if $\pi Q\equiv Q^*\pi =0$.   

\section{Computing stationary distributions for Markov chains and processes}\label{SectMP} 

Let $P$ be a transition matrix on the finite state space $V$.  A
point $y$ is said to be accessible from $x$ if $\{y\}$ is accessible from
$x$, and $x$ is said to communicate with $y$ if each is accessible from
the other.  It should be clear that communication is an equivalence
relation, and, for a given $x\in V$, we use $[x]$ to denote the communication class of
$x$ (i.e., the set of $y$ that communicate with $x$). Further, one says
that $x$ is recurrent or transient depending 
on whether $\mathbb P_x(\tau _x<\infty )=1$ or $\mathbb P(\tau_x<\infty)<1$.

The following facts are not hard and their proofs can be found in \S\,3.1 of~\cite{DWSMP}.  If
$x$ is recurrent and $y$ is accessible from $x$, then $y$ is also recurrent
and $x$ communicates with $y$.  Thus, no transient state is accessible from
any recurrent one. From~\eqref{returneq} with $\Delta =\emptyset $, one sees that if $x$
is transient, then the expected length of time that the chain starting at
$x$ spends at $x$ is finite.  Thus, at least one state in $V$ must be
recurrent, and every transient state must have a recurrent state to which it
is accessible.  

We next need to know that there always exists at least one stationary
distribution.  

\begin{lemma}\label{statlem}  If $P$ is a transition matrix and $x$ is a recurrent
  state, then $\mathbb E_x[\tau _x]<\infty $.  In addition, if, for $y\in V$,
$$\mu _y=\frac{1}{\mathbb E_x[\tau _x]} \mathbb E_x\left[\sum_{m=0}^{\tau
        _x-1}\1_{\{y\}}(X_m)\right],$$
then $\mu $ is a stationary distribution for $P$, and $\mu _y=0$ for
$y\notin [x]$.\end{lemma}
 
\begin{proof}  Since $P_{x'y}=0$ for $x'\in[x]$ and $y\notin[x]$, without
  loss in generality we may and will assume that $V=[x]$. By~\eqref{returneq}, $\mathbb E_x[\tau _x]<\infty $, and so it is clear
that $\mu $ is a probability distribution.  To prove that $\mu P=\mu $,
observe that
\begin{align*}
\mathbb E_x[\tau _x]\mu _y=\mathbb
E_x\left[\sum_{m=0}^{\tau_x-1}\1_{\{y\}}(X_m)\right] 
&=\mathbb E_x\left[\sum_{m=1}^{\tau_x}\1_{\{y\}}(X_m)\right] \\
&=\sum_{m=1}^\infty  \mathbb P_x\bigl(X_m=y\;\&\;\tau _x\ge m\bigr)\\
 &=\sum_{m=1}^\infty\sum_{z\in V}\mathbb P_x\bigl(X_{m-1}=z\;\&\;\tau _x>m-1\;\&\;X_m=y\bigr)\\
&=\sum_{z\in V}\sum_{m=0}^\infty \mathbb P_x\bigl(X_m=z\;\&\;\tau _x>m\bigr)P_{zy}\\
&=\sum_{z\in V}\mathbb E_z\left[\sum_{m=0}^{\tau_x-1}\1_{\{y\}}(X_m)\right]P_{zy}\\
&=\mathbb E_x[\tau _x](\mu P)_y,
\end{align*}
where, in the passage to the last line, we have used~\eqref{MP}.\end{proof}

We now have everything that we need to prove the following application of
Theorem~\ref{detthm2}.

\begin{theorem}\label{MPthm1}  If $Q=R(P-I)$ is a $Q$-matrix, where $R$ is
  a diagonal matrix with positive diagonal entries and $P$ is a transition
  matrix, then $\det(-Q^{\{x\}})\ge0$ for all $x\in V$.  Moreover, if
  the null space of $Q$ is one dimensional, then $\det(-Q^{\{x\}})>0$ if and
    only if $x$ is recurrent for $P$, and therefore
    $\rm{dim}\bigl(\rm{Null}(Q)\bigr)=1$ if and only if $\Pi_{-Q}>0$.
    Finally, if $\rm{dim}\bigl(\rm{Null}(Q)\bigr)=1$ and
\begin{equation}\label{pi} \pi _x=\frac{\det(-Q^{\{x\}})}{\Lambda _{-Q}}\quad\text{for }
x\in V,\end{equation}
then $\pi $ is the unique solution to $\mu Q=0$ satisfying $\langle \1,\mu
\rangle=1$. \end{theorem}

\begin{proof}  To prove that $\det(-Q^{\{x\}})$ is non-negative for all $x$
  and is positive when $x$ is recurrent for $P$, it suffices to handle the
  case when $R$ is the identity.  To that end, set $M=I-P$ and note that
  for any $\Delta \subsetneq V$ and $\alpha >0$, 
$$\bigl((\alpha I+M)^\Delta \bigr)^{-1}=\sum_{n=0}^\infty \alpha ^{-n-1}(P^\Delta )^n.$$
Hence $\bigl((\alpha I+M)^\Delta \bigr)^{-1}_{xx}>0$ for all $x\in
V\setminus \Delta $, and so by Theorem~\ref{detthm1}, $\det\bigl((\alpha
I+M)^{\{x\}}\bigr)>0$.  After letting $\alpha \searrow0$, it follows that
$\det(M^{\{x\}})\ge0$.

Next assume that $\rm{Null}(P-I)$ is one dimensional.  Then
$\rm{Null}\bigl((P-I)^*\bigr)$ is also one dimensional, and so there is only
one stationary distribution for the Markov chain determined by $P$.  Hence,
by Lemma~\ref{statlem} and the preceding discussion, every recurrent state must communicate with
every other one and be accessible from every transient state.  Moreover, if
$x$ is a recurrent state and $x\in\Delta \subsetneq V$, then Lemma~\ref{timelem} says that $(M^\Delta )^{-1}_{yy}<\infty $ for all $y\in
V\setminus \Delta $, and therefore, again by Theorem~\ref{detthm1},
$\det(M^{\{x\}})>0$.  In addition, since $\det(M^{\{x\}})\ge0$ for all $x$
and there must exist a recurrent $x$, it follows from~\eqref{eq4} that $\Pi _{-Q}>0$.
Finally, because $Q\1=0$, Theorem~\ref{detthm2} says that the right hand
side of~\eqref{pi} is the unique $\mu $ satisfying $\mu Q=0$ and $\langle
\1,\mu \rangle=1$.

What remains is to prove that $\det(-Q^{\{x\}})=0$ if $\Pi _{-Q}>0$ and $x$
is transient, and again it suffices to handle the case when $R=I$. 
But if $\Pi _{-Q}>0$, then we know that the $\pi $ in~\eqref{pi} satisfies
$\pi P=\pi $.  Thus, for any $x$ and $n\ge1$,
$$n\pi _x=\sum_{m=0}^{n-1}(\pi P^m)_x=\mathbb E_\pi
\left[\sum_{m=0}^{n-1}\1_{\{x\}}(X_m)\right],$$
where $\mathbb E_\pi $ denotes expectation with respect to the distribution
$\mathbb P_\pi $ of the Markov chain determined by $P$ with initial
distribution $\pi $.  Hence, 
$$n\pi _x\le\mathbb E_\pi \left[\sum_{m=0}^\infty \1_{\{x\}}(X_m)\right]$$
for all $n\ge1$.  Further, by~\eqref{MP2}, 
\begin{align*}
\mathbb E_\pi \left[\sum_{m=0}^\infty \1_{\{x\}}(X_m)\right]=\mathbb E_\pi
\left[\sum_{m=\xi ^{\{x\}}}^\infty \1_{\{x\}}(X_m)\right] 
&=\sum_{k=0}^\infty \sum_{m\ge k}^\infty \mathbb P_\pi \bigl(X_m=x\;\&\;\xi ^{\{x\}}=k\bigr)\\
&=\sum_{k=0}^\infty \sum_{m\ge k}\mathbb P_\pi (\xi ^{\{x\}}=k)\,\mathbb P_x(X_{m-k}=x)\\
&=\left(\sum_{k=0}^\infty \mathbb P_\pi (\xi ^{\{x\}}=k)\right)\mathbb
E_x\left[\sum_{m=0}^\infty\1_{\{x\}}(X_m)\right] \\
&=\mathbb P_\pi (\xi ^\Delta <\infty )\,\mathbb E_x\left[\sum_{m=0}^\infty \1_{\{x\}}(X_m)\right].
\end{align*}
Thus, if $x$ is transient and therefore 
$$\mathbb E_x\left[\sum_{m=0}^\infty
\1_{\{x\}}(X_m)\right]<\infty,$$ then $\pi _x$ and therefore
$\det(-Q^{\{x\}})$ must be $0$.\end{proof}  

\begin{remark}
Once one knows that there is some state that is accessible from every other
state, another proof that $\Pi _{-Q}>0$ can be based on the following
argument. Under these circumstances, Doeblin's theorem
(Theorem~2.2.1 of~\cite{DWSMP}) implies that 
$$\lim_{\alpha\downarrow 0} \alpha (\alpha I-Q)^{-1}v = \langle v, \pi \rangle$$
for all $v$. Now suppose that $0$ were not a simple eigenvalue of $Q$.
Then there would exist a $v \neq 0$ such that  $\langle v, \pi \rangle=0$
and $Q^nv=0$ for some $n \ge 1$. But this would mean that 
$$(\alpha I - Q)^{-1}v = \sum_{m=0}^n \alpha^{-m-1}Q^mv$$
and therefore that $Q^mv=0$ for all $m\ge 1$, from which is follows that 
$$v = \lim_{\alpha \downarrow 0}\alpha(\alpha I -Q)^{-1}v = \langle v, \pi \rangle=0.$$
Hence, $0$ must be a simple eigenvalue of $Q$. Furthermore,  because it is real, its non-real
eigenvalues come in conjugate pairs. Finally, if $-Q v = \lambda v$, then  
$e^{tQ}v=e^{-\lambda t}v$ and so, since $\|e^{tQ}v\|_{\rm u} \le
\|v\|_{\rm u}$, the real part of $\lambda$ must be non-negative. (Here $\|
\cdot \|_{\rm u}$ denotes the uniform norm.)
In particular, if $\lambda$ is real and different from $0$, it must be strictly positive. 
\end{remark}

\begin{corollary}\label{cor1} Assume that $Q = P-I$ where $P$ is a
  transition matrix with the property that the only solutions $v$ to
  $Pv=v$ are constant multiples of $1$.  Then, for any recurrent $x$,
\begin{equation}\label{coreq1}
\mathbb P_x(\tau _x \le \tau _y) = \frac{\det\bigl((I-
P)^{\{y\}}\bigr)}{\det\bigl((I-P)^{\{x,y\}}\bigr)}. 
\end{equation}
\end{corollary}

\begin{proof}  By Lemma~\ref{statlem}, one expression for the unique
  stationary distribution $\pi $ is
$$\pi _y=\frac{1}{\mathbb E_x[\tau _x]} \mathbb E_x\left[\sum_{m=0}^{\tau
        _x-1}\1_{\{y\}}(X_m)\right].$$
Hence, since 
$$\mathbb E_x\left[\sum_{m=0}^{\tau_x-1}\1_{\{x\}}(X_m)\right]=1,$$
we find
$$\frac{\pi _y}{\pi_x}=\mathbb E_x\left[\sum_{m=0}^{\tau_x-1}\1_{\{y\}}(X_m)\right].$$
By Theorem~\ref{MPthm1}, $\pi _y/\pi _x=\det\bigl((I-P)^{\{y\}}\bigr) / \det\bigl((I-P)^{\{x\}}\bigr)$.  At the same time, proceeding as in the
proof of Lemma~\ref{statlem}, one sees that
$$\mathbb E_x\left[\sum_{m=0}^{\tau _x-1}\1_{\{y\}}(X_m)\right]
=\mathbb P_x(\tau _y\le \tau _x)\mathbb E_y\left[\sum_{m=0}^{\tau
        _x-1}\1_{\{y\}}(X_m)\right].$$
Finally, by~\eqref{returneq} and~\eqref{eq1},
$$\mathbb E_y\left[\sum_{m=0}^{\tau
    _x-1}\1_{\{y\}}(X_m)\right]=\bigl((I-P)^{\{x\}}\bigr)^{-1}_{yy}=
\frac{\det\bigl((I-P)^{\{x,y\}}\bigr)}{\det\bigl((I-P)^{\{x\}}\bigr)}.$$
After combining these, we arrive at the asserted equation.\end{proof}

\begin{remark}  As a consequence of~\eqref{coreq1}, we see that
  $\det\bigl((I-P)^{\{y\}}\bigr)\le \det\bigl((I-P)^{\{x,y\}}\bigr)$ when
  $x$ is recurrent and $y\in V$.
This observation has the following generalization.  Namely, 
  given any transition matrix $P$ and $\Delta \subsetneq V$ and $y\in V\setminus \Delta $,
one has  that $\det\bigl((I-P)^{\Delta }\bigr)\le\det\bigl((I-P)^{\Delta
    \cup\{y\}}\bigr)$.  To see this, first observe that  
the argument with which we proved in
  Theorem~\ref{MPthm1} that $\det\bigl((I-P)^{\{x\}}\bigr)\ge0$ can be used
  to show that $\det\bigl((I-P)^{\Delta }\bigr)\ge0$ for any $\Delta
  \subseteq V$.  Thus there is nothing to do when $\Delta \subsetneq V$ and
  $\det\bigl((I-P)^{\Delta }\bigr)=0$.  On the other hand, if
  $\det\bigl((I-P)^{\Delta }\bigr)>0$, then one can use~\eqref{returneq}
  together with~\eqref{eq1} to identify the ratio of
  $\det\bigl((I-P)^{\Delta \cup\{y\}}\bigr)$ to $\det\bigl((I-P)^{\Delta
  }\bigr)$ as 
  $$\mathbb E_y\left[\sum_{m=0}^{\xi ^\Delta
    }\1_{\{y\}}(X_m)\right]\ge1.$$ By induction, this means that $\det\bigl((I-P)^{\Delta _1}\bigr)
\le \det\bigl((I-P)^{\Delta _2}\bigr)$ when $\Delta _1\subseteq\Delta _2$.\end{remark} 

In the case of a three-state Markov chain, it is easy to perform the calculations in Theorem~\ref{MPthm1} and Corollary~\ref{cor1}.

\begin{example}  Let $P=\bigl(p_{i,j}\bigr)_{1\le i,j\le 3}$ be a
  transition matrix. If we set
$$D(i,j,k)=p_{ji}(1-p_{kk})+p_{jk}p_{ki},$$
then
$$\det\bigl((I-P)^{\{1\}}\bigr)=D(1,2,3),\quad
  \det\bigl((I-P)^{\{2\}}\bigr)=D(2,3,1),\quad\text{and }
  \det\bigl((I-P)^{\{3\}}\bigr)=D(3,1,2).$$
Hence, by Theorem~\ref{detthm1}, $P$ has a unique stationary distribution
$\pi $ if and
only if
$$\Pi \equiv D(1,2,3)+D(2,3,1)+D(3.1,2)>0,$$
in which case,
$$\pi _1=\frac{D(1,2,3)}{\Pi },\quad
\pi _2=\frac{D(2,3,1)}{\Pi },\quad\text{and }
\pi _3=\frac{D(3,1,2)}{\Pi }.$$
Furthermore, if $D(1,2,3)>0$, then
$$\mathbb P_1(\tau _2<\tau _1)=\frac{D(2,3,1)}{1-p_{33}}\quad\text{and} \quad
\mathbb P_1(\tau _3<\tau _1)=\frac{D(3,1,2)}{1-p_{22}},$$
and similarly for $\mathbb P_2$ and $\mathbb P_3$.
\end{example}

\section{Wilson's Algorithm and Kirchhoff's Formula}
 
Let $\Gamma =(V,E)$ be a connected $N$-vertex graph in which no vertex has an edge to
itself and any two vertices are connected by at most one edge.  Recall that the graph Laplacian
matrix $\mathcal L=\mathcal D-\mathcal A$, where $\mathcal A$ is the
adjacency matrix for $\Gamma $ and $\mathcal D$ is the diagonal matrix of
degrees.  Think of $-\mathcal L$ as a $Q$-matrix.  Clearly, the
connectedness of $\Gamma $ implies that all the vertices communicate with
one another and are therefore all recurrent.  Hence, by 
Theorem~\ref{MPthm1}, $\det\bigl(\mathcal L^\Delta \bigr)>0$ for all
non-empty $\Delta \subsetneq V$.  In addition, if $(x_1,\dots,x_N)$ is
any ordering of the vertices, then, by Theorem~\ref{detthm1},
\begin{equation}\label{Lprod}\frac1{\det\bigl(\mathcal L^{\{x_1\}}\bigr)}=\prod_{m=1}^{N-1}
\bigl(\mathcal L^{\{x_1,\dots,x_m\}}\bigr)^{-1}_{x_{m+1}x_{m+1}}.\end{equation}
At the same time, since $\pi \mathcal
L=0$ when $\pi _x=1/N$ for all $x\in V$, Theorem~\ref{MPthm1} says that
$$\det\bigl(\mathcal L^{\{x\}}\bigr)=\frac{\Pi _{\mathcal L}}N
\quad\text{for all } x\in V.$$
By combining these, we see that, for any ordering 
$(x_1,\ldots,x_N)$ of the elements of $V$,
\begin{equation}\label{Kirk}
\prod_{m=1}^{N-1}\bigl(\mathcal L^{\{x_1,\dots,x_m\}}\bigr)^{-1}_{x_{m+1}x_{m+1}}=
\frac1{\det\bigl(\mathcal L^{\{x_1\}}\bigr)}=\frac N{\Pi _{\mathcal L}}.\end{equation}

In order to explain the relevance of the preceding to Wilson's algorithm,
we have to explain what his algorithm is.  
Set $P=\mathcal D^{-1}\mathcal A$.  Given an ordering $(x_1,\ldots,x_N)$ of the
vertices, take $\Delta _1=\{x_N\}$ and run a Markov chain $\{X_n:\,n\ge1\}$
with transition matrix $P$ starting from $x_1$.  Consider the segment
$\{X_n:\,0\le n\le\xi ^{\Delta _1}\}$, and let $(Y_{1,1},\dots,Y_{1,K_1})$
be the successive points visited by the path obtained from this segment by
erasing all of its loops 
(i.e., cycles).  Note that both $K_1$ and  $Y_{1,k}$, $1 \le k \le K_1$, will be random.  If $\{Y_{1,1},\dots,Y_{1,K_1}\} =V$,
then $(Y_{1,1},\dots,Y_{1,K_1})$ is a spanning tree with a single branch
running from $x_1$ to $x_N$, in which case the algorithm
terminates.  If $K_1<N$, set $\Delta _2=\{Y_{1,1},\dots,Y_{1,K_1}\}$, and
take $x_{2,1}$ to be the first vertex from $(x_1,\dots,x_N)$
that is not in $\Delta _2$.  Run the Markov chain starting from $x_{2,1}$,
consider the segment $\{X_n:\,0\le n\le \xi ^{\Delta _2}\}$, and let $(Y_{2,1},\dots,
Y_{2,K_2})$ be the successive vertices visits by its loop erasure.  If
$\{Y_{1,1},\dots,Y_{1,K_1}\} \cup \{Y_{2,1},\dots,Y_{2,K_2}\} =V$, then again the algorithm stops and
$\{Y_{1,1},\dots,Y_{1,K_1}\}\cup\{Y_{2,1},\dots,Y_{2,K_2}\}$ are the
vertices of a spanning tree that has two branches if $Y_{2,K_2}=Y_{1,K_1}$ and three
branches otherwise.  One continues running the algorithm in this way until it produces a
spanning tree.  Since $\mathbb P_x(\xi ^\Delta <\infty )=1$ for all $x\in
V$ and $\Delta \subsetneq V$, with probability $1$ it will produce a tree
after no more than $N$ runs.

Wilson's theorem says that the probability of his algorithm producing any
particular spanning tree is the same for all spanning trees.  In view of
the preceding description, proving his theorem comes down to the
following computation.  Let $\emptyset \neq\Delta \subsetneq V$ be given, and assume that $(y_1,\dots,y_K)$ is a given
ordering of $K$ distinct vertices from $V\setminus \Delta$.

\begin{question} What is the probability $\mathcal P^\Delta (y_1,\dots,y_K)$ that $(y_1,\dots,y_K)$ will
be the successive points visited by  the loop erasure of the segment $\{X_n:\,0\le n\le\xi ^\Delta \}$ of the
Markov chain started at $y_1$?  
\end{question}

If one knows the answer, then one knows how to compute the
probability $\mathcal P(\mathcal T)$ that Wilson's algorithm
 produces the spanning tree $\mathcal T$. Indeed, suppose that the branch structure of $\mathcal T$ is
$$(y_{1,1},\dots,y_{1,K_1})\cdots(y_{L,1},\dots,y_{L,K_L}),$$ where
\begin{enumerate}
\item for each $1\le \ell\le L$, the vertices $y_{\ell,1}, \cdots, y_{\ell,K_\ell}$ are distinct,
  and 
  $$V=\{y_{\ell,k}:\,1\le\ell\le L\text{ and } 1\le k\le K_L\},$$
\item and for each $2\le \ell\le L$,
$$\{y_{\ell,1},\ldots,y_{\ell,K_{\ell}}\}\cap\bigcup_{j=1}^{\ell-1}\{y_{j,1},\ldots,y_{j,K_{j}}\}
=\{y_{\ell,K_\ell}\}.$$
\end{enumerate}
Then
$$\mathcal P(\mathcal T) =\prod_{\ell=1}^L\mathcal P^{\Delta _\ell}(y_{\ell,1},\dots,y_{\ell, K_\ell}),$$
where $\Delta _{1}=\{y_{1,K_1}\}$ and $\Delta _{\ell}=\Delta
_{\ell-1}\cup\{y_{\ell,1},\dots,y_{\ell,K_\ell}\}$ for $2\le \ell\le
L$.

 Hence, if we can show that
$$\mathcal P^\Delta (y_1,\dots,y_K)=\prod_{k=1}^{K-1}\bigl(\mathcal
L^{\Delta \cup\{y_j:\,1\le j<k\}}\bigr)^{-1}_{y_{k}y_{k}},$$
then we will know that
\begin{equation}\label{Lawler}
\mathcal P(\mathcal T)= \prod_{\ell=1}^L
\prod_{k=1}^{K_\ell-1}\bigl(\mathcal
L^{\Delta_\ell \cup\{y_{\ell,j}:\,1\le j<k\}}\bigr)^{-1}_{y_{\ell,k}y_{\ell,k}},
\end{equation}
which in conjunction with~\eqref{Kirk} would mean that
\begin{equation}\label{Wilson}
\mathcal P(\mathcal T)=\frac1{\det(\mathcal L^{\{y_{1,K_1}\}})}=\frac N{\Pi
  _{\mathcal L}}.\end{equation}

The equality in~\eqref{Lawler} was proved by Lawler.  To derive it, for $\mathbf{m}=(m_1,\dots,m_{K-1})\in\mathbb{N}^K$, let $\sigma ^{(\mathbf{m})}_1$ be the time of
the $m_1$th visit of $\{X_{n}:\,n\ge 0\}$ to $y_1$, and, for
$2\le k<K$, let $\sigma ^{(\mathbf{m})}_k$ be the time of the $m_k$th visit
of $\{X_n:\,n\ge \sigma ^{(\mathbf{m})}_{k-1}\}$ to $y_k$.  Also, define ${\sigma '}
^{(\mathbf{m})}_1$ to be time of the $(m_1+1)$th visit of $\{X_{n}:\,n\ge0\}$ to
$y_1$ and ${\sigma'}^{(\mathbf{m})}_k$ to be the time of the $(m_k+1)$th visit of
$\{X_n:\,n\ge \sigma ^{(\mathbf{m})}_{k-1}\}$ to $y_k$.  Then
$$\mathcal P^\Delta (y_1,\dots,y_K)=\sum_{\mathbf{m}\in\mathbb{N}^{K-1}}\mathbb P_{y_1}(B^{(\mathbf{m})}_{K-1}),$$
where, for any $1\le k<K$,
$$B_k^{(\mathbf{m})}\equiv\bigl\{\sigma ^{(\mathbf{m})}_j<\zeta ^\Delta <{\sigma '}^{(\mathbf{m})}_j\;\&\;
X_{\sigma ^{(\mathbf{m})}_j+1}=y_{j+1}\text{ for } 1\le j\le k\bigr\}.$$

To compute $\mathbb P_{y_1}(B_k^{\mathbf{m}})$, define $\zeta ^{(\mathbf{m})}_1=\zeta ^\Delta
$ and, for $1\le k< K$, let $\zeta ^{(\mathbf{m})}_k$ be the time of the
first visit of $\{X_n:\,n\ge\sigma ^{(\mathbf{m})}_{k-1}\}$ to $\Delta _k\equiv\Delta
\cup\{y_1,\dots,y_{k-1}\}$.  Then 
$$B_k^{(\mathbf{m})}=\bigl\{\sigma ^{(\mathbf{m})}_j<\zeta ^{(\mathbf{m})}_j\;\&\;X_{\sigma ^{(\mathbf{m})}_j+1}=y_{j+1}\text{ for } 1\le j\le k\bigr\}.$$
Using~\eqref{MP2}, one can show that $\mathbb P_{y_1}(B_1)=d_{y_1}^{-1}\mathbb P_{y_1}(\tau
^{(m_1)}_{y_1}<\zeta ^\Delta )$, where $\tau^{(m)}_x$ is the time of the
$m$th visit to $x$ by $\{X_n:\,n\ge0\}$.  For $2\le k\le K$, note that
$$B_k^{(\mathbf{m})}=B^{(\mathbf{m})}_{k-1}\cap\bigl\{\sigma ^{(\mathbf{m})}_k<\zeta ^{(\mathbf{m})}_k\;\&\;
X_{\sigma ^{(\mathbf{m})}_k+1}=y_{k+1}\bigr\},$$
and again use~\eqref{MP2} to see that
$$\mathbb P_{y_1}(B^{(\mathbf{m})}_k)=d_{y_k}^{-1}\mathbb P_{y_k}\bigl(\tau 
^{(m_k)}_{y_k}<\zeta ^{\Delta _k}\bigr) 
\mathbb P_{y_1}\bigl(B^{(\mathbf{m})}_{k-1}\bigr).$$
By induction, this shows that
$$\mathbb P_{y_1}\bigl(B^{(\mathbf{m})}_K\bigr)=\prod_{k=1}^K\frac{\mathbb P_{y_k}\bigl(\tau
^{(m_k)}_{y_k}<\zeta ^{\Delta _k}\bigr)}{d_{y_k}},$$
and therefore, after summing over $\mathbf{m}\in\mathbb{N}^K$ and applying~\eqref{returneq},
we have that
$$\mathcal P^\Delta (y_1,\dots,y_K)=\prod_{k=1}^{K-1}
\frac{\bigl((I-P)^{\Delta _k}\bigr)^{-1}_{y_{k+1},y_{k+1}}}{d_{y_k}}=
\prod_{k=1}^{K-1}\bigl(\mathcal L^{\Delta _k}\bigr)^{-1}_{y_{k+1},y_{k+1}}$$

The equation~\eqref{Wilson} contains a proof that Wilson's algorithm yields the
uniform distribution on the spanning trees.  In addition, since $\mathbb
P(\mathcal T)$ is the reciprocal of the number of spanning trees,
Kirchhoff's formula follows at once as well, namely,
\begin{equation}\label{Kirkform}
\text{the number of spanning trees in $\Gamma $}=\det\bigl(\mathcal L^{\{x\}}\bigr)=
\frac{\Pi _{\mathcal L}}N.\end{equation}

\begin{examplecont}
Suppose that we implement Wilson's algorithm to generate a spanning tree of
$\Gamma$. We start with our original graph as shown in the picture on the
left having vertex set  
$\{x_1,x_2,x_3,x_4,x_5,x_6\}$. Start a simple random walk at $x_1$ and stop
it when it first reaches $x_6$. Assume the loop-erasure of this path is
$[x_1, x_3, x_6]$,  and add this branch to the spanning tree. For the
second branch, since $x_2$ is the first vertex in $V$ not visited in the
first branch, start a random walk at $x_2$ and stop it when it reaches 
$\{x_1, x_3, x_6\}$. Assume that the loop-erasure of this path is
$[x_2,x_3]$ and add this branch to the spanning tree.  Finally, start a
simple random walk at $x_4$ and stop it when it reaches  $\{x_1, x_3, x_6\}
\cup \{x_2,x_3\}$. Assume that the loop-erasure of this path is
$[x_4,x_5,x_6]$, and add this third branch to the spanning tree. 
   This completes the construction of the spanning tree shown below on the
   right.  In the general notation from the proof above, we have
  $[y_{1,1}, y_{1,2},y_{1,3}] = [x_1, x_3, x_6]$, $[y_{2,1},y_{2,2}] =
 [x_2,x_3]$,  $[y_{3,1},y_{3,2},y_{3,3}] = [x_4,x_5,x_6]$, and $\Delta_1 =
 \{y_{1,K_1}\}=\{y_{1,3}\}=\{x_6\}$. 
We will now check that the probability that Wilson's algorithm produces
this spanning tree is $1/29$, which is the reciprocal of the number of
spanning trees of $\Gamma$. 
\begin{center}
\begin{minipage}{\textwidth}
  \centering
  \raisebox{-0.5\height}{\includegraphics[height=1.35in]{newfig1.pdf}}
  \hspace*{.3in}
  \raisebox{-0.5\height}{\includegraphics[height=1.35in]{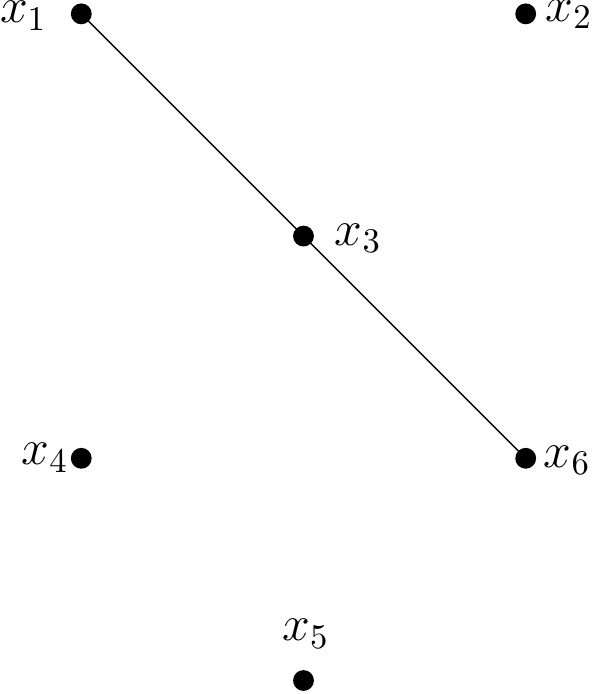}}
  \hspace*{.3in}
  \raisebox{-0.5\height}{\includegraphics[height=1.35in]{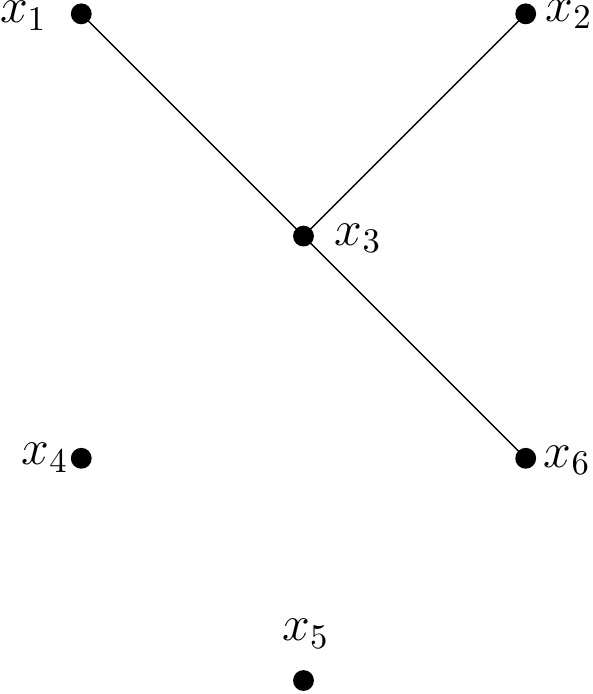}}
  \hspace*{.3in}
  \raisebox{-0.5\height}{\includegraphics[height=1.35in]{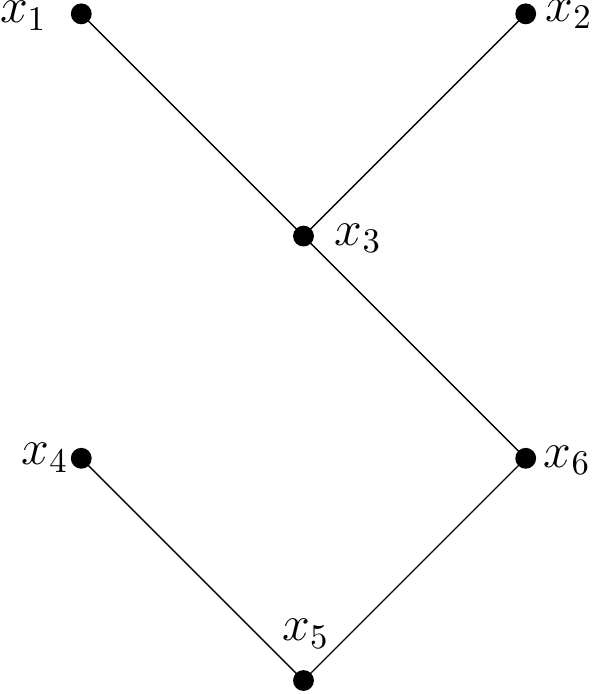}}
\end{minipage}
\end{center}
We know from~\eqref{Wilson} that the probability of this particular spanning tree being generated is
$$\mathcal{P}(\mathcal T) = \frac{1}{\det[ \glm^{\{x_6\}}]} = \frac{\det\left[\bigl((I-P)^{\{x_6\}}\bigr)^{-1}\right] }{\det \left[ \degm^{\{x_6\}}\right]}
=\det\left[\bigl((I-P)^{\{x_6\}}\bigr)^{-1}\right] \prod_{j=1}^5 \frac{1}{d_{x_j}} .
$$
Since the order that the remaining vertices were added to the spanning is $\{x_1,x_3, x_2,x_4,x_5\}$,  we know from Theorem~\ref{detthm1} that
\begin{align*}
\det\left[\bigl((I-P)^{\{x_6\}}\bigr)^{-1}\right] = 
\bigl((I-P)^{\{x_6\}}\bigr)^{-1}_{x_1x_1} &\cdot \bigl((I-P)^{\{x_6,x_1\}}\bigr)^{-1}_{x_3x_3}
\cdot \bigl((I-P)^{\{x_6,x_1,x_3\}}\bigr)^{-1}_{x_2x_2} \cdot \\
&\;\;\cdot \bigl((I-P)^{\{x_6,x_1,x_3,x_2\}}\bigr)^{-1}_{x_4x_4}
\cdot \bigl((I-P)^{\{x_6,x_1,x_3,x_2,x_4\}}\bigr)^{-1}_{x_5x_5}.
\end{align*}
The transition matrix for simple random walk on the graph $\Gamma$ is
$$P=
\kbordermatrix{
&x_1 &x_2 &x_3  &x_4 &x_5 &x_6 \\
  x_1 &0  &1/3 &1/3 &1/3  &0 &0\\
  x_2 &1/2  &0 &1/2 &0  &0 &0\\
  x_3 &1/4   &1/4 &0 &1/4  &0 &1/4\\
  x_4 &1/3 &0 &1/3 &0  &1/3 &0\\
  x_5 &0   &0 &0 &1/2  &0 &1/2\\
  x_6 &0 &0 &1/2 &0  &1/2 &0
},$$
and so we find  $\bigl((I-P)^{\{x_6\}}\bigr)^{-1}$ equals
$$\left(\kbordermatrix{
&x_1 &x_2 &x_3  &x_4 &x_5\\
  x_1 &1  &-1/3 &-1/3 &-1/3  &0\\
  x_2 &-1/2  &1 &-1/2 &0  &0\\
  x_3 &-1/4   &-1/4 &1 &-1/4  &0\\
  x_4 &-1/3 &0 &-1/3 &1  &-1/3\\
  x_5 &0   &0 &0 &-1/2  &1
}\right)^{-1}
\!=\kbordermatrix{
&x_1 &x_2 &x_3  &x_4 &x_5\\
  x_1 &93/29  &50/29   &76/29 &60/29  &20/29\\
  x_2 &75/29   &74/29  &80/29 &54/29  &18/29\\
  x_3 &57/29   &40/29   &84/29 &48/29  &16/29\\
  x_4 &60/29      &36/29   &64/29 &78/29  &26/29\\
  x_5 &30/29      &18/29 &32/29 &39/29  &42/29}.$$
Moreover, $\bigl((I-P)^{\{x_6,x_1\}}\bigr)^{-1}$ equals
$$\left(\kbordermatrix{
&x_2 &x_3  &x_4 &x_5\\
  x_2 &1 &-1/2 &0  &0\\
  x_3 &-1/4 &1 &-1/4  &0\\
  x_4 &0 &-1/3 &1  &-1/3\\
  x_5 &0 &0 &-1/2  &1}\right)^{-1}
=\kbordermatrix{
&x_2 &x_3  &x_4 &x_5\\
  x_2 &36/31 &20/31 &6/31  &2/31\\
  x_3 &10/31 &40/31 &12/31  &4/31\\
  x_4 &4/31   &16/31 &42/31  &14/31\\
  x_5 &2/31   &8/31   &21/31  &38/31},
$$
$$\bigl((I-P)^{\{x_6,x_1,x_3\}}\bigr)^{-1}=
\left(\kbordermatrix{
&x_2   &x_4 &x_5\\
  x_2 &1  &0  &0\\
  x_4 &0  &1  &-1/3\\
  x_5 &0 &-1/2  &1
  }\right)^{-1}
=\kbordermatrix{
&x_2   &x_4 &x_5\\
  x_2 &1  &0  &0\\
  x_4 &0  &6/5  &2/5\\
  x_5 &0 &3/5  &6/5},
$$
$$\bigl((I-P)^{\{x_6,x_1,x_3,x_2\}}\bigr)^{-1}
=\kbordermatrix{
   &x_4 &x_5\\
  x_4  &6/5  &2/5\\
  x_5 &3/5  &6/5}, \;\;\; \text{and} \;\;\;\bigl((I-P)^{\{x_6,x_1,x_3,x_2,x_4\}}\bigr)^{-1}=\kbordermatrix{
          &x_5 \\
x_5   &1}, $$
so that
$$\det\left[\bigl((I-P)^{\{x_6\}}\bigr)^{-1}\right] = \frac{93}{29} \cdot
\frac{40}{31} \cdot 1 \cdot \frac{6}{5} \cdot 1  = \frac{144}{29}.$$ 
Hence,
$$\mathcal{P}(\mathcal T) 
=\frac{144}{29} \cdot \frac{1}{3}\cdot \frac{1}{2}\cdot \frac{1}{4}\cdot
\frac{1}{3}\cdot \frac{1}{2} 
=\frac{1}{29}.$$
But we already knew that there are 29 spanning trees of $\Gamma$, so we
have verified in this case that Wilson's algorithm does, in fact, produce a
spanning tree uniformly at random. 
\end{examplecont}

\section{Cayley's theorem}\label{SectCayley}

If $\Gamma=(V,E)$ is the complete graph on $N+1$ vertices so that there is an edge connecting each vertex to every other, then Cayley's theorem states that the number of spanning trees of $\Gamma$ is $(N+1)^{N-1}$. This formula was first discovered  in 1860 by Carl Wilhelm Borchardt, although it is now universally named after Arthur Cayley~\cite{Cayley} who extended Borchardt's original results. Of course, Cayley's theorem is easily derived from Kirchhoff's matrix tree theorem by computing $\det[\glm]$ using elementary column operations to bring the matrix into lower triangular form.

The easiest way to prove Cayley's theorem via Wilson's algorithm is to use~\eqref{returneq}.
Start a simple random walk at $x$, and suppose that $\D \subseteq V\setminus\{x\}$ is any nonempty collection of vertices  with $|\D|=m$.
Recall that $r_\D(x)$ is the probability that simple random walk starting at $x$ returns to $x$ before entering $\D$. Let $r_{\D}(x;k)$ be the probability that simple random walk starting at $x$ returns to $x$ in exactly $k$ steps without entering $\D$ so that
$$r_\D(x) = \sum_{k=2}^\infty r_\D(x;k)$$
because it takes the simple random walk at least 2 steps to return to its starting point.
Since each vertex has an edge to every other vertex, we have partitioned the vertex set into three pieces, namely $V_1=\{x\}$, $V_2=\D$ which has cardinality $m$, and $V_3$ which has cardinality $N-m$. Thus, the probability that simple random walk starting at $x$ returns to $x$ in exactly $k$ steps without entering $\D$ is
$$r_\D(x;k) = \mathbb P_{x}\{ S_1\in V_3, S_2\in V_3, \ldots, S_{k-1}\in V_3, S_k = x\} = \frac{N-m}{N}\left(\frac{N-1-m}{N}\right)^{k-2} \frac{1}{N}$$
and so
$$r_\D(x) = \frac{(N-m)}{N^2}\sum_{k=2}^\infty \left(\frac{N-1-m}{N}\right)^{k-2}  = \frac{N-m}{N(m+1)}.$$
By~\eqref{returneq},
\begin{equation}\label{cayleyeq1}
\bigl((I-P)^\Delta \bigr)^{-1}_{xx}= \frac{1}{1-r_\D(x)} = \frac{N(m+1)}{m(N+1)}.
\end{equation}
Suppose that we now label the vertices of $\Gamma$ as $V=\{y_1,\ldots, y_{N+1}\}$, and set  $\D_j=\{y_1, \ldots, y_{j}\}$ for $j=1,\ldots, N$. Since $|\D_j|=j$, we have from Theorem~\ref{detthm1} combined with~\eqref{cayleyeq1},
$$\det\left[\bigl((I-P)^{\{y_1\}}\bigr)^{-1}\right] =\prod_{j=1}^{N} \frac{N(j+1)}{j(N+1)} = \frac{N^N(N+1)!}{(N+1)^{N}N!}= \frac{N^N}{(N+1)^{N-1}}.$$
Since each of the $(N+1)$ vertices has degree  $N$, we conclude that the number of spanning trees of the complete graph on $N+1$ vertices is
$$\det\bigl(\glm^{\{y_1\}}\bigr) = \frac{\det\bigl(\degm^{\{y_1\}}\bigr)}{\det\left[\bigl((I-P)^{\{y_1\}}\bigr)^{-1}\right]} = \frac{N^N}{\frac{N^N}{(N+1)^{N-1}}} = (N+1)^{N-1}.$$

\begin{remark}
 Lyons and Peres~\cite{LP} use Wilson's algorithm to prove Cayley's theorem.  However, their proof of Wilson's algorithm is via cycle-popping and so their derivation is quite different than ours. Moreover, they do not extend their proof of  Cayley's theorem to a proof of Kirchhoff's matrix tree theorem. 
 \end{remark}

\section*{Acknowledgements}
The first author learned that Lawler's proof of Wilson's algorithm could be used to deduce Kirchhoff's matrix tree theorem during a talk that Greg Lawler gave at
the Mathematical Sciences Research Institute (MSRI) in Berkeley during Spring 2012, and so he would like to thank both MSRI for the invitation to spend a semester at the institute, and Greg Lawler for explaining his proof of Wilson's algorithm.  Very special thanks are owed to Shlomo Sternberg whose keen interest in this project and encouragement throughout proved invaluable. Finally, several people answered questions and helped clarify points while we were preparing this paper, and so we would like to thank them: Shaun Fallat, Tony Guttmann, and Aihua Xia. 

\bibliographystyle{plain}

\end{document}